\documentclass[12pt, twoside, leqno]{article}

\usepackage{amsmath,amsfonts,amssymb,url,amsthm,mathtools,tikz,enumitem}
\usepackage[utf8]{inputenc}

\newtheorem{proposition}{Proposition}[section]
\newtheorem{lemma}[proposition]{Lemma}
\newtheorem{corollary}[proposition]{Corollary}

\newtheorem{theorem}[proposition]{Theorem}

\newtheorem{example}[proposition]{Example}

\newcommand{\eps}{\varepsilon}
\newcommand{\ph}{\varphi}

\newcommand\C{{\mathbb C}}

\newcommand\Q{{\mathbb Q}}
\newcommand\R{{\mathbb R}}
\newcommand\Z{{\mathbb Z}}

\newcommand{\norm}{{\mathcal{N}}}
\newcommand{\OO}{{\mathcal{O}}}

\newcommand{\tilc}{\tilde{c}}
\newcommand{\tilf}{\tilde{f}}

\newcommand{\height}{{\mathrm{h}}}

\newcommand{\ord}\nu

\newcommand{\gerp}{{\mathfrak{p}}}
\newcommand{\gerP}{{\mathfrak{P}}}

\newcommand{\bfa}{{\mathbf a}}

\numberwithin{equation}{section}

\title{Binary polynomial power sums  vanishing at roots of unity}

\author{Yuri Bilu\\
IMB, Université de Bordeaux \& CNRS\\
E-mail: yuri@math.u-bordeaux.fr,
\and
Florian Luca\\
School of Maths, Wits University, South Africa\\
Research Group in Algebraic Structures and Applications,\\ 
King Abdulaziz University, Jeddah, Saudi Arabia\\
Centro de Ciencias Matemáticas,\\ 
UNAM, Morelia, Mexico\\
E-mail: Florian.Luca@wits.ac.za}

\date{}

1

\makeatletter

\renewcommand*\l@section[2]{%
  \ifnum \c@tocdepth >\z@
    \addpenalty\@secpenalty
    \addvspace{0.2em \@plus\p@}%
    \setlength\@tempdima{1.5em}%
    \begingroup
      \parindent \z@ \rightskip \@pnumwidth
      \parfillskip -\@pnumwidth
      \leavevmode \bfseries
      \advance\leftskip\@tempdima
      \hskip -\leftskip
      #1\nobreak\hfil \nobreak\hb@xt@\@pnumwidth{\hss #2}\par
    \endgroup
  \fi}

\makeatother

\numberwithin{equation}{section}

\begin{document}

\baselineskip=17pt

\hbadness 300

\hfuzz 5pt

    \makeatletter
    \let\@fnsymbol\@alph
    \makeatother

\maketitle

\begin{abstract}
Let ${c_1(x),c_2(x),f_1(x),f_2(x)}$ be  polynomials with rational coefficients. 
The ``obvious'' exceptions being excluded, there can be at most finitely many roots of unity among the zeros of the polynomials  ${c_1(x)f_1(x)^n+c_2(x)f_2(x)^n}$ with ${n=1,2\ldots}$. We estimate the orders of these roots of unity in terms of the degrees and the heights of the polynomials $c_i$ and $f_i$.



2020 \emph{Mathematics Subject Classification}: Primary 11D61; Secondary 11G50, 11J86.

\emph{Key words and phrases}: polynomial power sums, roots of unity, primitive divisors.

\end{abstract}

{\scriptsize

\tableofcontents

}

\section{Introduction}
\label{sintr}

Let $c_1(x),c_2(x),f_1(x),f_2(x)$ be non-zero polynomials in ${\mathbb Q}[x]$. We denote by ${\bf u}:=\{u_n(x)\}_{n\ge 1}\subset {\mathbb Q}[x]$ 
 the sequence of polynomials given by 
\begin{equation}
\label{eq:1}
u_n(x)=c_1(x)f_1(x)^n+c_2(x)f_2(x)^n\quad {\text{\rm for~all}}\quad n\ge 1.
\end{equation}
We study roots of unity $\zeta$ such that ${u_n(\zeta)=0}$ for some~$n$. It can happen accidentally that  $u_n(x)$ is the zero polynomial for some~$n$. We ignore these~$n$. 
We would like to show that aside from some exceptional situations, the following holds true:
there exist at most finitely many roots of unity~$\zeta$ such that for some~$n$ the polynomial $u_n(x) $ is not identically zero but ${u_n(\zeta)=0}$.

The following example shows that we indeed have to exclude some exceptional cases. 

\begin{example}
\label{exinf}
Let $a,b$ be integers with $b$ non-zero,
and assume that 
$$
c_2(x)/c_1(x)=\delta x^a, \qquad f_2(x)/f_1(x)=\eps x^b, \qquad \delta,\eps\in \{1,- 1\}.
$$ 
We then get
$$
u_n(x)=c_1(x)f_1(x)^n(1+\delta \eps^n x^{a+bn})
$$
and we see that if $x=\zeta$ is such that $\zeta^{a+bn}=-\delta\eps^n$, then ${u_n(\zeta)=0}$.
The condition that $b\ne 0$ insures that $u_n(x)$ is non-zero for $n$ sufficiently large (in fact, for all $n$ except eventually one of them, namely $n=-a/b$), and every $u_n(x)$ vanishes at the roots of unity of order ${|a+bn|}$ or ${2|a+bn|}$ depending on the sign of $\delta\eps^n$. 
\end{example}

It turns out that this example is the only case when the polynomials  $u_n(x)$ vanish at infinitely many roots of unity. We have the following theorem.

\begin{theorem} 
\label{thm:thmmain}
Let ${c_1(x),c_2(x),f_1(x),f_2(x)\in \Q[x]}$ be non-zero polynomials. For a positive integer~$n$   define $u_n(x)$ as in~\eqref{eq:1}. Then the following two conditions are equivalent.
\begin{enumerate}
\item
\label{iinf}
There exist infinitely many roots of unity~$\zeta$ such that for some~$n$ the polynomial $u_n(x) $ is not identically zero but ${u_n(\zeta)=0}$.

\item
\label{iex}
There exist ${a,b \in \Z}$ with ${b\ne 0}$ and ${\delta,\eps \in \{1,-1\}}$ such that 
$$
c_2(x)/c_1(x)=\delta x^a, \qquad f_2(x)/f_1(x)=\eps x^b.
$$
\end{enumerate}
\end{theorem}

It is not hard to derive this theorem from classical results on unlikely intersection like the Theorem of Bombieri-Masser-Zannier-Maurin~\cite{BHMZ10,BMZ99,Ma08}. See also the recent work of Ostafe and Shparlinski~\cite{Os16,OS20},  especially Theorem~2.11 and Corollary~2.14   in~\cite{OS20}. 

However, we are mainly interested in a quantitative statement: when  condition~\ref{iex} of Theorem~\ref{thm:thmmain} is not satisfied, we want to bound the orders of the roots of unity~$\zeta$ such that ${u_n(\zeta)=0}$ for some~$n$, in terms of the degrees and the heights of our polynomials $f_i,c_i$. To the best of our knowledge, no quantitative version of the Bombieri-Masser-Zannier-Maurin theorem is available which would imply such a bound. 

To state our result, let us recall the definition of the height of a non-zero polynomial  in $\Q[x]$. The height of a primitive vector ${\bfa=(a_1, \ldots,a_k)\in \Z^{k}}$ 
(\textit{primitive} means that ${\gcd(a_1,\ldots, a_k)=1}$) is defined by 
$$
\height(\bfa):=\log\max\{|a_1|, \ldots, |a_k|\}. 
$$
In general, given a non-zero vector ${\bfa\in \Q^{k+1}}$, there exists ${\lambda \in \Q^\times}$, well defined up to multiplication by~$\pm1$,  such that ${\bfa^\ast=\lambda \bfa}$ is primitive, and we set ${\height(\bfa):=\height(\bfa^\ast)}$. 

We define the height of a non-zero polynomial ${g(x)\in \Q[x]}$ as the height of the vector of its  coefficients. More generally, we define the height of a non-zero vector ${(g_1, \ldots, g_k)\in \Q[x]^k}$ as the height of the vector formed of the  coefficients of all polynomials $g_1, \ldots, g_k$. 

We have the following theorem.

\begin{theorem} 
\label{thm:thmmain1}
Let ${c_1(x),c_2(x),f_1(x),f_2(x)\in \Q[x]}$ be non-zero polynomials such that  condition~\ref{iex} of Theorem~\ref{thm:thmmain} is not satisfied. 
Set
\begin{align*}
D&:=\max\{\deg c_1,\deg c_2,\deg f_1, \deg f_2\},\\
X&:=\max\{3,\height(c_1,c_2),\height(f_1,f_2)\}.
\end{align*}
Let~$m$ be a positive integer and~$\zeta$  a primitive $m$th root of unity such that  for some~$n$ the polynomial $u_n(x) $ is not identically zero but ${u_n(\zeta)=0}$. Then
\begin{equation}
\label{eupperm}
m\le e^{100D(X+D)}.  
\end{equation}
\end{theorem}

The numerical constant~$100$ here is rather loose; probably, one can replace it by~$4$ or so.  
  
One may ask whether there is a bound for~$m$ which depends only on one of the parameters~$D$ or~$X$. The following examples show that this is not the case.

\begin{example}
Consider 
${u_n(x)=(2x)^n-2^m}$, 
for which
$$ 
(c_1(x),c_2(x),f_1(x),f_2(x))=(1,-2^m,2x,1), \qquad X=\max\{3,m\log 2\}. 
$$
Then
${u_m(x)=2^m(x^m-1)}$ 
vanishes at primitive $m$th roots of unity, and we have ${m\ge X/\log2}$ (provided ${m\ge 5}$). Hence no bound independent of~$X$ is possible. 
\end{example}

\begin{example}
\label{exd}
Consider
${u_n(x)=x^n+x^D+1}$, 
for which
$$
(c_1(x),c_2(x),f_1(x),f_2(x))=(1,x^D+1,x,1). 
$$
Then  $u_{2D}=(x^{3D}-1)/(x^D-1)$ vanishes at primitive $3D$th roots of unity, so we have ${m\ge 3D}$. Hence no bound independent of~$D$ is possible. 
\end{example}

One may also ask whether in Theorem~\ref{thm:thmmain1} one can bound~$n$ such that $u_n(x)$ vanishes at a root of unity. The answer is ``no'' in general. Indeed, if polynomials ${c_1(x)f_1(x)}$ and ${c_2(x)f_2(x)}$ have a common root, then every $u_n(x)$ will vanish at that root. But even if ${c_1(x)f_1(x)}$ and ${c_2(x)f_2(x)}$ do not simultaneously vanish at some root of unity, it is still possible that $u_n(x)$ vanishes at a root of unity for infinitely many~$n$. This is, for instance, the case for the sequence ${u_n(x)=x^n+x^D+1}$ from Example~\ref{exd}: it vanishes at  primitive $3D$th roots of unity whenever ${n\equiv 2D\bmod 3D}$. Nevertheless, we can bound the \textit{smallest}~$n$ with this property. Here is the precise statement.

\begin{theorem}
\label{thsmallestn}
In the set-up of Theorem~\ref{thm:thmmain}, assume that, for a given~$m$, the set of positive integers~$n$ with the property ``the polynomial $u_n(x)$ is not identically~$0$ but vanishes at an $m$th root of unity'' is not empty. Then the smallest~$n$ in this set satisfies 
$$
n \le m(\log m)^3(X+\log D). 
$$
More precisely, either there exists~$n$ in this set satisfying ${n\le 2m}$, or every~$n$ in this set satisfies ${n\le m(\log m)^3(X+\log D)}$. 
\end{theorem}

Throughout the article we use standard notation. We denote  $\ph(n)$  the Euler function, $\mu(n)$ the Möbius function, $\Lambda(n)$ the von Mangoldt function  and ${\omega(n)}$  the number of prime divisors of~$n$ counted without multiplicities. 

Theorems~\ref{thm:thmmain} and~\ref{thm:thmmain1} are proved in Section~\ref{sproofs}, and Theorem~\ref{thsmallestn} is proved in Section~\ref{ssmallestn}.  
In Sections~\ref{sheights},~\ref{scyclo} and~\ref{sprim} we collect various auxiliary facts used in the proof. In particular, in Section~\ref{sprim} we revisit Schinzel's classical Primitive Divisor Theorem~\cite{Sc74}. We obtain a version of this theorem fully explicit in all parameters, which is key ingredient in our proof of Theorem~\ref{thm:thmmain1}.

\section{Heights}
\label{sheights}

All results of this section are well-known, but sometimes we prefer to give a short proof than to look for a bibliographical reference. 

Recall the definition of the absolute logarithmic (projective) height. Let 
$$
\bar\alpha=(\alpha_0,\alpha_1, \ldots, \alpha_k)\in \bar\Q^{k+1}
$$
 be a non-zero vector of algebraic numbers. Pick a number field~$K$ containing all~$\alpha_i$ and normalize the absolute values of~$K$ to extend the standard absolute values of~$\Q$. With this normalization, the height 
 of~$\bar\alpha$ is defined by 
\begin{equation}
\label{eheight}
\height(\bar\alpha)=d^{-1}\sum_{v\in M_K}d_v\log\max\{|\alpha_0|_v,\ldots, |\alpha_k|_v\},  
\end{equation}
where ${d=[K:\Q]}$ and ${d_v=[K_v:\Q_v]}$ is the local degree. This definition is known to be independent of the choice of~$K$ and invariant under multiplication of~$\bar\alpha$ by a non-zero algebraic number: ${\height(\lambda\bar\alpha)=\height(\bar\alpha)}$ for ${\lambda \in \bar\Q^\times}$. When ${\alpha \in \Q^{n+1}}$ this definition coincides with the definition of height from Section~\ref{sintr}. 

Separating the contributions of infinite and finite places, we can rewrite equation~\eqref{eheight}  as 
\begin{equation}
\label{eheightold}
\begin{aligned}
\height(\bar\alpha)&= d^{-1}\sum_{K\stackrel\sigma\hookrightarrow \C} \log  \max\{|\alpha_0^\sigma|,\ldots, |\alpha_k^\sigma|\}\\
&+ d^{-1}\sum_{\gerp}\max\{-\nu_\gerp(\alpha_0), \ldots, -\nu_\gerp(\alpha_k)\}\log\norm\gerp,
\end{aligned}
\end{equation}
where the first sum is over the complex embeddings of~$K$, the second sum is over the finite primes of~$K$, and $\norm\gerp $ denotes the absolute norm of~$\gerp$. 

Now we define the height $\height(g)$ of a non-zero polynomial~$g$ with algebraic coefficients (in one or in several variables), or, more generally, the height ${\height(g_1,\ldots, g_k)}$ of a  vector of such polynomials as the height of the vector of all coefficients of those polynomials (ordered somehow).

With a standard abuse of notation, for ${\alpha\in \bar\Q}$ we write ${\height(\alpha)}$ for ${\height(1,\alpha)}$. If~$\alpha$ belongs to a number field~$K$ then 
\begin{align}
\label{ehplus}
\height(\alpha)&=d^{-1}\sum_{v\in M_K}d_v\log^+|\alpha|_v\\ 
\label{ehmin}
&=d^{-1}\sum_{v\in M_K}-d_v\log^-|\alpha|_v \qquad (\alpha\ne 0), 
\end{align}
where ${\log^+=\max\{\log, 0\}}$ and  ${\log^-=\min \{\log,0\}}$. 

\begin{lemma}
\label{lhpol}
Let ${\alpha\in \bar\Q}$ and ${f(x)\in\bar\Q[x]}$ a polynomial of degree less or equal to~$D$. Then 
\begin{equation}
\label{ehfa}
\height(f(\alpha))\le D\height(\alpha)+ \height(1,f)+\log(D+1).
\end{equation}
More generally, if ${g(x)\in\bar\Q[x]}$ is another polynomial of degree less or equal to~$D$ and ${g(\alpha) \ne0}$ then
\begin{equation}
\label{ehfga}
\height(f(\alpha)/g(\alpha))\le D\height(\alpha)+ \height(g,f)+\log(D+1).
\end{equation}
If ${f(\alpha) =0}$ then 
\begin{equation}
\label{ehroot}
\height(\alpha)\le \height(f)+\log 2. 
\end{equation}
Furthermore, let~$r$ be a non-negative integer. Then
\begin{equation}
\label{ehder}
\height(1,f^{(r)}/r!) \le \height(1,f)+D\log2. 
\end{equation}
\end{lemma}

\begin{proof}
We start by proving~\eqref{ehfga}. By definition,
$$
\height(f(\alpha)/g(\alpha))=\height(1,f(\alpha)/g(\alpha))=\height(g(\alpha), f(\alpha)). 
$$
Write 
$$
f(x)=a_Dx^D+\cdots+a_0, \qquad g(x)=b_Dx^D+\cdots+b_0. 
$$ 
Let~$K$ be a number field containing~$\alpha$ and the coefficients of $f,g$. We set ${d=[K:\Q]}$.  For  ${v\in M_K}$ we have 
$$
|f(\alpha)|_v  \le
\begin{cases}
(D+1)|f|_v\max\{1,|\alpha|_v\}^D,& v\mid \infty,\\
|f|_v|\max\{1,|\alpha|_v\}^D,& v<\infty, 
\end{cases}
$$
where ${|f|_v =\max\{|a_0|_v, \ldots,|a_D|_v\}}$, and similarly for $g(\alpha)$. 
Hence 
\begin{align*}
\height(g(\alpha), f(\alpha))&\le d^{-1}\sum_{v\in M_K}d_v\log\max\{|g(\alpha)|_v,|f(\alpha)|_v\}\\
&\le d^{-1}\sum_{v\in M_K}d_v(\log\max\{|f_v|,|g|_v\}+D\log^+|\alpha|_v) \\
&+d^{-1}\sum_{\substack{v\in M_K\\v\mid \infty}}d_v\log(D+1)\\
&= \height(g,f)+D\height(\alpha)+ \log(D+1), 
\end{align*}
which proves~\eqref{ehfga}. 

For~\eqref{ehroot}  see \cite[Proposition~3.6(1)]{BB13}. Finally, we have 
$$
\frac{f^{(r)}}{r!}(x)=\sum_{k=r}^D\binom kr a_k x^{k-r}. 
$$
Since 
$$
\binom kr \le 2^k\le 2^D, 
$$
we have 
$$
\left|\frac{f^{(r)}}{r!}\right|_v \le 
\begin{cases}
2^D|f|_v, & v\mid \infty,\\
|f|_v, & v<\infty. 
\end{cases}
$$
Hence 
\begin{align*}
\height\left(1, \frac{f^{(r)}}{r!}\right) &= d^{-1}\sum_{v\in M_K}d_v\log^+\left|\frac{f^{(r)}}{r!}\right|_v\\
&\le d^{-1}\sum_{v\in M_K}d_v\log^+|f_v| +d^{-1}\sum_{\substack{v\in M_K\\v\mid \infty}}d_vD\log2\\
&= \height(1,f)+D\log2.  
\end{align*}
The lemma is proved.
\end{proof}

\begin{lemma}
\label{lhdivide}
Let ${f_1(x), \ldots, f_k(x)\in \bar \Q[x]}$ be non-zero polynomials of degrees not exceeding~$D$, and let ${g(x) \in \bar\Q[x]}$ be a common divisor of ${f_1, \ldots, f_k}$ (in the ring ${\bar\Q[x]}$). Then 
$$
\height(f_1/g, \ldots, f_k/g) \le \height(f_1, \ldots, f_k) +(D+k-1)\log2. 
$$
\end{lemma}

\begin{proof}
Consider the polynomial
$$
f(x,y_1, \ldots, y_{k-1}):=f_1(x)y_1+\cdots +f_{k-1}(x)y_{k-1}+ f_k(x) \in \bar\Q[x,y_1, \ldots, y_{k-1}]. 
$$
Applying Theorem~1.6.13 from \cite{BG06}, we obtain  
$$
\height(f/g) \le \height(f/g)+\height(g) \le \height(f) +(D+k-1)\log2.
$$
Since 
$$
\height(f_1/g, \ldots, f_k/g)=  \height(f/g), \qquad \height(f_1, \ldots, f_k)=\height(f), 
$$
the result follows. 
\end{proof}

\begin{lemma}
\label{lvals}
Let~$K$ be a number field of degree~$d$ and ${\alpha\in K}$. Then 
\begin{equation}
\label{evalsone}
\sum_{\nu_\gerp(\alpha)<0}\log\norm\gerp \le d\height(\alpha), \qquad \sum_{\nu_\gerp(\alpha)>0}\log\norm\gerp \le d\height(\alpha),
\end{equation}
where the first sum is over (finite) primes~$\gerp$ of~$K$ with ${\nu_\gerp(\alpha)<0}$, the second sum over those with ${\nu_\gerp(\alpha)>0}$, and in the second sum we assume ${\alpha\ne 0}$. 
More generally, let ${\alpha_1, \ldots, \alpha_k\in K}$. Then 
\begin{equation}
\label{evalsmany}
\sum_{\substack{\nu_\gerp(\alpha_i)<0\ \text{for} \\\text{some}\ i\in \{1,\ldots,k\}}}\log\norm\gerp \le d\height(\bar\alpha),
\qquad \bar\alpha=(1,\alpha_1, \ldots, \alpha_k). 
\end{equation}
\end{lemma}

\begin{proof}
Inequality~\eqref{evalsmany} is immediate from~\eqref{eheightold} (note that ${\alpha_0=1}$), and both statements in~\eqref{evalsone} are special cases of~\eqref{evalsmany}. 
\end{proof}

\begin{lemma}[``Liouville's inequality'']
\label{lliouv}
Let~$K$ and~$\alpha$ be as in Lemma~\ref{lvals}, ${\alpha \ne 0}$.  Let  ${S\subset M_K}$ be any set of places of~$K$ (finite or infinite). Then 
$$
e^{-d\height(\alpha)}\le \prod_{v\in M_K}|\alpha|_v^{d_v}\le e^{d\height(\alpha)}. 
$$
In particular, if ${\sigma_1, \ldots \sigma_r:K\hookrightarrow\C}$ are some distinct complex embeddings of~$K$ then 
$$
\prod_{i=1}^r |\alpha^{\sigma_i}|\ge e^{-d\height(\alpha)}. 
$$
\end{lemma}
We omit the proof, which is well-known and easy.

\section{Cyclotomic polynomials}
\label{scyclo}

We denote ${\Phi_m(T)}$ the $m$th cyclotomic polynomial. We will systematically use the identity
\begin{equation}
\label{ecyclomu}
\Phi_m(T) = \prod_{d\mid m}(T^d-1)^{\mu(m/d)},
\end{equation}

In this section we study values of cyclotomic polynomials at algebraic points. We give an asymptotic expression for the height of $\Phi_m(\gamma)$ as ${\gamma \in \bar\Q}$ is fixed and ${m\to\infty}$. We also estimate the absolute value of $\Phi_m(\gamma)$ from below. 

The results of this section can be viewed as totally explicit versions of some results from \cite[Section~3]{BBL13}, and we follow~\cite{BBL13} rather closely. We note however that all this goes back to the 1974 work of Schinzel~\cite{Sc74} or even earlier.

\subsection{The height}

\begin{theorem}
\label{tas}
Let~$\gamma$ be an algebraic number. Then
$$
\height(\Phi_m(\gamma))=\ph(m)\height(\gamma)+O_1\bigl(2^{\omega(m)}\log (\pi m)\bigr).
$$
\end{theorem}

Recall that ${A=O_1(B)}$ means that ${|A|\le B}$.

To prove this theorem we need some preparations. We follow \cite[Section~3]{BBL13} with some changes. 


\begin{proposition}
\label{pcyc}
For a positive integer~$m$ we have 
\begin{equation}
\label{ephinz}
\max_{|z|\le1}\log|\Phi_m(z)|\le 2^{\omega(m)}\log (\pi m),
\end{equation}
the maximum being over the unit disc on the complex plane. (We use the convention ${\log0=-\infty}$.) For ${0<\eps\le 1/2}$ we also have 
\begin{equation}
\label{ephitriv}
\min_{|z|\le1-\eps}\log|\Phi_m(z)|\ge -2^{\omega(m)}\log  \frac1\eps. 
\end{equation}
\end{proposition}

\begin{proof}
By the maximum principle, it suffices to prove that~\eqref{ephinz} holds for complex~$z$ with ${|z|=1}$. Thus, fix such~$z$. We will actually prove a slightly sharper bound 
\begin{equation}
\label{ecircle}
\log|\Phi_m(z)|\le (2^{\omega(m)-1}+1)\log m+2^{\omega(m)}\log\pi. 
\end{equation}  
We can write~$z$ in a unique way as ${z=\zeta e^{2\pi i\theta/m}}$, where~$\zeta$ is an $m$th root of unity (not necessarily primitive) and ${-1/2<\theta\le 1/2}$. We may assume ${\theta\ne 0}$, because for the finitely many~$z$ with ${\theta=0}$ the bound extends by continuity.

Let~$\ell$ be the exact order of~$\zeta$; thus, ${\ell\mid m}$  and~$\zeta$  is a primitive $\ell$th root of unity. Let~$d$ be any other divisor of~$m$. If ${\ell\nmid d}$ then ${d\le m/2}$ and
$$
2\ge |z^d-1| \ge 2\sin(\pi d/2m)\ge 2d/m.
$$ 
(We use the inequality ${|\sin x|\ge (2/\pi)x}$ which holds for ${|x|\le\pi/2}$.)  This implies that 
\begin{equation}
\label{elndivd}
 \bigl|\log |z^d-1| \bigr|\le\log (m/d).  
\end{equation}
And if ${\ell\mid d}$ then we have
${|z^m-1| = 2\sin(\pi  \theta d/m)}$, which implies that
$$
2\pi\theta d/m \ge |z^d-1|\ge 4\theta d/m. 
$$
Writing   ${d=d'\ell}$,
this implies that 
\begin{equation}
\label{eldivd}
\log |z^{d'\ell}-1| = \log d'-\log \frac m{2\ell\theta}+O_1\left(\log\pi\right).  
\end{equation}
Using~\eqref{ecyclomu} we obtain
\begin{align*}
\log|\Phi_m(z)|
&=  \sum_{\genfrac{}{}{0pt}{}{d\mid m}{\ell\nmid d}} \mu\left(\frac md\right) \log |z^d-1|+  \sum_{d'\mid m/\ell}\mu\left(\frac{m/\ell}{d'}\right) \log |z^{\ell d'}-1|\\
&\le \sum_{d\mid m} \left|\mu\left(\frac md\right)\right| \log \frac md+ \sum_{d'\mid m/\ell}\mu\left(\frac{m/\ell}{d'}\right) \left(\log d'-\log \frac m{2\ell\theta}\right)\\
&\hphantom{\le{}}+O_1(2^{\omega(n/\ell)}\log\pi)\\
&=2^{\omega(m)-1}\sum_{p\mid m}\log p +\Lambda\left(\frac m\ell\right)   +\delta\log(2\theta)+ O_1(2^{\omega(m/\ell)}\log\pi), 
\end{align*}
where   ${\delta=0}$ if ${\ell<m}$ and ${\delta=1}$ if ${\ell=m}$. Since ${\log(2\theta)\le 0}$, 
this proves~\eqref{ecircle}.

The proof of~\eqref{ephitriv} is much easier. 
When ${|z|\le 1-\eps}$, we have 
$$
2\ge |z^d-1|\ge 1-|z|^d\ge 1-|z|\ge \eps. 
$$
Since ${0<\eps\le 1/2}$ this implies that ${\bigl|\log|z^d-1|\bigr|\le \log (1/\eps)}$. We obtain 
$$
\bigl|\log |\Phi_m(z)|\bigr|=\left|\sum_{d\mid m}\mu\left(\frac md\right) \log|z^d-1|\right| \le 2^{\omega(m)}\log \frac1\eps.
$$
In particular,~\eqref{ephitriv} holds.  
\end{proof}

\begin{corollary}
\label{ccyc}
Let~$m$ be a positive integer and ${z\in \C}$. Then 
$$
\log^+|\Phi_m(z)|= \ph(m)\log^+|z|+O_1\bigl(2^{\omega(m)}\log (\pi m)\bigr),
$$
where ${\log^+=\max\{\log, 0\}}$. 
\end{corollary}

\begin{proof}
For ${|z|\le 1}$ this is  Proposition~\ref{pcyc}. If ${ |z|>1}$ then
\begin{equation}
\label{ephire}
\log|\Phi_m(z)|=\ph(m)\log|z|+\log|\Phi_m(z^{-1})|, 
\end{equation}
and ${\log|\Phi_m(z^{-1})|\le 2^{\omega(m)}\log (\pi m)}$ by Proposition~\ref{pcyc}. This already implies the upper bound 
$$
\log^+|\Phi_m(z)|\le  \ph(m)\log^+|z|+2^{\omega(m)}\log (\pi m). 
$$
The lower bound 
\begin{equation}
\label{ephilo}
\log^+|\Phi_m(z)|\ge  \ph(m)\log^+|z|-2^{\omega(m)}\log (\pi m)
\end{equation}
is trivial when ${m=1}$, so we will assume ${m\ge 2}$ in the sequel. In the case ${1<|z|\le m/(m-1)}$ we have 
$$
\log^+|\Phi_m(z)| \ge 0 \ge \ph(m)\log\frac{m}{m-1}-1\ge \ph(m)\log^+|z|-1,
$$
which is much better than wanted. Finally, if ${|z|\ge m/(m-1)}$, then 
$$
\log|\Phi_m(z^{-1})|\ge -2^{\omega(m)}\log m
$$ 
by~\eqref{ephitriv} with ${\eps=1/m}$. Hence~\eqref{ephilo} follows from~\eqref{ephire} in this case.   
\end{proof}

\paragraph{Proof of Theorem~\ref{tas}.}
We  use~\eqref{ehplus} with ${\alpha=\Phi_m(\gamma)}$. For ${v\in M_K}$ we have 
$$
\log^+|\Phi_m(\gamma)|_v=
\begin{cases}
\ph(m)\log^+|\gamma|_v+O_1\bigl(2^{\omega(m)}\log (\pi m)\bigr), & v\mid \infty,\\
\ph(m)\log^+|\gamma|_v, & v<\infty.
\end{cases}
$$
Indeed, the archimedean case is Corollary~\ref{ccyc}, and the non-archimedean case is obvious. Summing up, the result follows. 
\qed

\subsection{The lower bound}

The following result is proved in \cite[Corollary~4.2]{BL20} as a consequence of Baker's theory of logarithmic forms. 

\begin{proposition}
\label{pabs}
Let~$\gamma$ be a complex algebraic number of degree~$d$, not a root of unity, and~$n$ a positive integer. Then 
\begin{equation*}
|\gamma^n-1|\ge 
e^{-10^{12}d^4(\height(\gamma)+1)\log (n+1)}. 
\end{equation*}
\end{proposition}

\begin{corollary}
\label{carch}
Let~$\gamma$ and~$m$ be as in Proposition~\ref{pabs}. Then 
\begin{equation}
\label{elowerreal}
\log |\Phi_m(\gamma)|\ge -10^{12}d^4(\height(\gamma)+1)\cdot 2^{\omega(m)}\log (m+1).
\end{equation}
\end{corollary}

\begin{proof}
If ${|\gamma|\ge 1}$ then 
$$
\log|\Phi_m(\gamma)|=\ph(m)\log|\gamma|+\log|\Phi(\gamma^{-1})|\ge \log|\Phi(\gamma^{-1})|. 
$$
Hence, replacing, if necessary,~$\gamma$ by~$\gamma^{-1}$, we may assume ${|\gamma|\le 1}$. 
We have 
\begin{equation}
\label{esumagain}
\log|\Phi_m(\gamma)|=\sum_{n\mid m}\mu\left(\frac mn\right) \log |\gamma^n-1|. 
\end{equation}
Proposition~\ref{pabs} implies that 
$$
2\ge |\gamma^n-1|\ge e^{-10^{12}d^4(\height(\gamma)+1)\log (n+1)}. 
$$
Hence for ${1\le n\le m}$ we have 
$$
\bigl|\log |\gamma^n-1|\bigr|\le 10^{12}d^4(\height(\gamma)+1)\log (m+1). 
$$
Substituting this to~\eqref{esumagain}, we obtain 
$$
\bigl|\log |\Phi_m(\gamma)|\bigr|\le 10^{12}d^4(\height(\gamma)+1)\cdot2^{\omega(m)}\log (m+1).
$$
In particular, we proved~\eqref{elowerreal}. 
\end{proof}

\section{Schinzel's Primitive Divisor Theorem}
\label{sprim}
Let~$\gamma$ be a non-zero algebraic number, not a root of unity. We consider the sequence 
$$
u_n=u_n(\gamma)=\gamma^n-1. 
$$
(Note that in this section $(u_n)$ is a numerical sequence, while in the other sections it is a sequence of polynomials.) A prime $\gerp$ of the number field ${K=\Q(\gamma)}$ is called \textit{primitive divisor} for~$u_n$ if 
$$
\nu_\gerp(u_n) >0, \qquad \nu_\gerp(u_k)=0 \quad (k=1, \ldots, n-1). 
$$
For further use, let us fix here some basic properties of primitive divisors. Recall that ${\Phi_n(T)}$ denotes the $n$th cyclotomic polynomial, and~$\norm\gerp$ is the absolute norm of~$\gerp$. 

\begin{proposition}
\label{pprim}
Assume that~$\gerp$ is a primitive divisor of~$u_n$. Then~$n$ divides ${\norm\gerp-1}$ and ${\nu_\gerp(\Phi_n(\gamma))\ge 1}$. In particular, ${n<\norm\gerp}$. 
\end{proposition}

The proofs are very easy and we omit them.

Schinzel~\cite{Sc74} proved that~$u_n$ admits a primitive divisor for ${n\ge n_0(d)}$, where~$d$ is the degree of~$\gamma$. This was an improvement upon the earlier work~\cite{PS68}, where the same was proved under the assumption ${n\ge n_0(\gamma)}$. 

Stewart~\cite{St77} made Schinzel's result explicit, but he imposed an additional hypothesis ${\gamma=\alpha/\beta}$, where ${\alpha,\beta\in \OO_K}$ are coprime algebraic integers. Here we obtain a fully explicit version of Schinzel's result without any extra hypothesis.

\begin{theorem}
\label{thschin}
Let~$\gamma$ be an algebraic number of degree~$d$, not a root of unity. Assume that 
\begin{equation}
\label{ehypo}
n\ge \max\{2^{d+1},10^{30}d^{9}\}.
\end{equation}
Then ${u_n=\gamma^n-1}$ admits a primitive divisor. 
\end{theorem}


Theorem~\ref{thschin} is a consequence of the following result, appearing, albeit in a different setting, in Schinzel's work.

\begin{proposition}
\label{pup}
In the above set-up, assume that~$u_n$ does not admit a primitive divisor. Then 
\begin{equation}
\label{eup}
\height(\Phi_n(\gamma)) \le 10^{13}d^4     
(\height(\gamma)+1)\cdot2^{\omega(n)}\log (n+1).
\end{equation}
\end{proposition}

\subsection{Proof of Proposition~\ref{pup}}
We start from the following well-known fact.

\begin{lemma}
\label{lwellknown}
Let~$K$ be a number field of degree~$d$ and~$p$ a prime number. Let~$\gerp$ be a prime of~$K$ above~$p$ of ramification index~$e_\gerp$ (that is, ${e_\gerp=\nu_\gerp(p)}$). Let ${\xi\in K}$ satisfy 
$$
\nu_\gerp(\xi-1) >\frac {e_\gerp}{p-1}. 
$$
Then for any positive integer~$n$ we have
$$
\nu_\gerp(\xi^n-1)=\nu_\gerp(\xi-1)+\nu_\gerp(n). 
$$
\end{lemma}

The proof of the lemma can be found, for instance, in \cite[Lemma~1]{PS68}.

\begin{lemma}
\label{lschin}
Let~$\gamma$ be an algebraic number of degree~$d$, not a root of unity, and~$n$ an integer satisfying ${n\ge 2^{d+1}}$. Let~$\gerp$ be a prime  of the field $\Q(\gamma)$ which is not a primitive divisor of ${u_n=\gamma^n-1}$. Then 
${\nu_\gerp(\Phi_n(\gamma))\le \nu_\gerp(n)}$. 
\end{lemma}

This is Schinzel's~\cite{Sc74} crucial ``Lemma~4''. Since his set-up is slightly different, we reproduce the proof here.

\begin{proof}
We may assume that ${\nu_\gerp(\gamma^n-1)>0}$, since there is nothing to prove otherwise. In particular, ${\nu_\gerp(\gamma)=0}$.

For ${k=0,1,2\ldots}$ denote~$\ell_k$ the multiplicative order of ${\gamma \bmod \gerp^k}$; that is,~$\ell_k$ is the smallest positive integer~$\ell$ with the property ${\nu_\gerp(\gamma^\ell- 1)\ge k}$. Clearly, ${\nu_\gerp(\gamma^n -1)\ge k}$ if and only if ${\ell_k\mid n}$. Together with~\eqref{ecyclomu} this implies that for every~$k$ the following holds:
\begin{equation}
\label{eschinzel}
\nu_\gerp\bigl(\Phi_n(\gamma)\bigr)= \sum_{i=1}^k \sum_{\ell_i\mid m\mid n}\mu\left(\frac nm\right)+\sum_{\ell_{k+1}\mid m\mid n}\mu\left(\frac nm\right)\bigl(\nu_\gerp(\gamma^m-1) -k\bigr)
\end{equation}
Let~$p$ be the rational prime below~$\gerp$ and ${e_\gerp=\nu_\gerp(p)}$ the ramification index. We will apply~\eqref{eschinzel} with 
$$
k=\left\lfloor \frac {e_\gerp}{p-1}\right\rfloor, 
$$
which will be our choice of~$k$ from now on. 
We claim that
\begin{equation}
\label{eclaim}
n>\ell_{k+1}.
\end{equation}
We postpone the proof of~\eqref{eclaim} (which is a bit messy) until later, and now complete the proof of the lemma assuming  validity of~\eqref{eclaim}. 

Since ${n>\ell_{k+1}\ge\ell_i}$ for ${i=1, \ldots, k}$, the double sum in~\eqref{eschinzel} vanishes. Also, 
if ${\ell_{k+1}\mid m}$ then
$$
\nu_\gerp(\gamma^m-1)=\nu_\gerp(\gamma^{\ell_{k+1}}-1)+\nu_\gerp\left(\frac{m}{\ell_{k+1}}\right) 
$$
by Lemma~\ref{lwellknown}. Hence~\eqref{eschinzel} can be rewritten as 
\begin{equation}
\label{eschinzelbis}
\nu_\gerp\bigl(\Phi_n(\gamma)\bigr)= \sum_{\ell_{k+1}\mid m\mid n}\mu\left(\frac nm\right)\bigl(\nu_\gerp(\gamma^{\ell_{k+1}}-1) -k\bigr)+ \sum_{\ell_{k+1}\mid m\mid n}\mu\left(\frac nm\right)\nu_\gerp\left(\frac{m}{\ell_{k+1}}\right). 
\end{equation}
Since ${n>\ell_{k+1}}$, the first sum in~\eqref{eschinzelbis} vanishes. As for the second sum, it vanishes (just being empty) if ${\ell_{k+1}\nmid n}$. From now on assume that ${\ell_{k+1}\mid n}$ and set ${n'=n/\ell_{k+1}}$. We obtain
\begin{equation*}
\nu_\gerp\bigl(\Phi_n(\gamma)\bigr)= e_\gerp\sum_{m'\mid n'}\mu\left(\frac {n'}{m'}\right)\nu_p\left(m'\right)=
\begin{cases}
e_\gerp,& \text{$n'$ is a power of~$p$},\\
0,& \text{otherwise}.
\end{cases}
\end{equation*}
In any case we obtain ${\nu_\gerp\bigl(\Phi_n(\gamma)\bigr)\le \nu_\gerp(n)}$. This proves the lemma. 

We are left with the claim~\eqref{eclaim}. Note first of all that 
\begin{equation}
\label{eellone}
n>\ell_1 
\end{equation}
because~$\gerp$ is not a primitive  divisor of~$u_n$.
Another useful observation is that 
\begin{equation}
\label{epelli}
\ell_{i+1}\le p\ell_i\qquad (i=1,2, \ldots).
\end{equation}
Indeed, 
$$
\gamma^{p\ell_i}-1=\sum_{j=1}^{p-1}\binom pj(\gamma^{\ell_i}-1)^j+(\gamma^{\ell_i}-1)^p,
$$
which implies that ${\nu_\gerp(\gamma^{p\ell_i}-1)>\nu_\gerp(\gamma^{\ell_i}-1)}$, proving~\eqref{epelli}. 

If ${k=0}$ then~\eqref{eclaim} is~\eqref{eellone}. Now assume that ${k\ge 1}$. In this case 
\begin{equation}
\label{ebet}
p-1\le e_\gerp\le d.
\end{equation}
On the other hand, let ${p^{f_\gerp}=\norm\gerp}$ be the absolute norm of~$\gerp$. Clearly, 
$$
\ell_1\le p^{f_\gerp}-1\le p^{d/e_\gerp}-1.
$$
In the special case ${p=3}$, ${e_\gerp=d=2}$ we have ${k=1}$ and 
${\ell_2\le p\ell_1 \le 6}$. 
Since ${n\ge 2^{d+1}=8}$ by the hypothesis, this proves~\eqref{eclaim} in this special case.
From now on we assume that ${d\ge 3}$ for ${p=3}$.  

Using~\eqref{epelli} iteratively, we obtain 
$$
\ell_{k+1}\le p^k\ell_1 <p^{e_\gerp/(p-1)+d/e_\gerp}\le \max_{p-1\le t\le d}p^{t/(p-1)+d/t}=p^{1+d/(p-1)}.
$$
We have to show that 
\begin{equation*}
p^{1+d/(p-1)}\le 2^{d+1}.
\end{equation*}
This is true by inspection in the cases 
$$
p=2, \qquad p=3,\ d\ge 3, \qquad p=5, \ d\ge 4.
$$
Now assume that ${p\ge 7}$, in which case ${d\ge 6}$. Since ${p\le d+1}$, we have 
$$
p^{1+d/(p-1)} \le (d+1) \cdot7^{d/6}. 
$$
A calculation shows that 
${(d+1) \cdot7^{d/6} \le 2^{d+1}}$
for ${d\ge 6}$. This completes the proof of~\eqref{eclaim}. 
\end{proof}

\paragraph{Proof of Proposition~\ref{pup}}
We use~\eqref{ehmin} with ${\alpha=\Phi_n(\gamma)}$. 
For ${v\in M_K}$ we have 
$$
-\log^-|\Phi_n(\gamma)|_v\le
\begin{cases}
10^{12}d^4(\height(\gamma)+1)\cdot2^{\omega(n)}\log (n+1), & v\mid \infty,\\
-\log|n|_v, & v<\infty.
\end{cases}
$$
Indeed, the archimedean case is Corollary~\ref{carch}, and the non-archimedean case is Lemma~\ref{lschin}. Summing up, we obtain
$$
\height(\Phi_n(\gamma))\le 10^{12}d^4(\height(\gamma)+1)\cdot2^{\omega(n)}\log (n+1)+\log n,
$$
which is sharper than~\eqref{eup}. 
\qed

\subsection{Proof of Theorem~\ref{thschin}}

Assume~$u_n$ does not have a primitive divisor, but~$n$ satisfies~\eqref{ehypo}. We have, in particular, ${n\ge 10^{30}}$. 
Comparing  Proposition~\ref{pup} and Theorem~\ref{tas}, we obtain
\begin{align*}
\ph(n)\height(\gamma)&\le 
10^{13}d^4     %
(\height(\gamma)+1)\cdot2^{\omega(n)}\log (n+1)+2^{\omega(n)}\log (\pi n)\\
&\le 
10^{14}d^4     
(\height(\gamma)+1)\cdot2^{\omega(n)}\log (n+1). 
\end{align*}  
Since~$\gamma$ is not a root of unity, we have 
\begin{equation}
\label{evout}
d\height(\gamma) \ge 2(\log(3d))^{-3},
\end{equation}
see \cite[Corollary~2]{Vo96}. Hence 
$$
\ph(n)\height(\gamma)\le  
10^{15}d^5    (\log(3d))^3 
\height(\gamma)\cdot2^{\omega(n)}\log (n+1),
$$
which implies 
\begin{equation}
\label{ealmost}
\ph(n)\le  
10^{15}d^5    (\log(3d))^3 
\cdot2^{\omega(n)}\log (n+1).
\end{equation}
For ${n\ge 10^{30}}$ we have 
\begin{equation}
\label{eboundphomega}
\ph(n) \ge 0.5 \frac n{\log\log n},\qquad
\omega(n)\le \frac{\log n}{\log\log n-1.2}, 
\end{equation}
see \cite[Theorem~15]{RS62} and~\cite[Theorem~13]{Ro83}. 
Hence for ${n\ge 10^{30}}$ 
\begin{align*}
2^{\omega(n)}\frac{n}{\ph(n)}\log(n+1) &\le   n^{(\log2)/(\log\log(10^{30})-1.2)} \cdot 2(\log\log n) \cdot \log(n+1)\\
&\le n^{1/3}. 
\end{align*}
Using this, we deduce from~\eqref{ealmost} the inequality 
${n^{2/3}\le 10^{15}d^5    (\log(3d))^3 }$. 
A quick calculation shows that this inequality is incompatible with~\eqref{ehypo}. \qed

\section{Proof of Theorems~\ref{thm:thmmain} and~\ref{thm:thmmain1}}
\label{sproofs}


Since condition~\ref{iex} of Theorem~\ref{thm:thmmain} trivially implies condition~\ref{iinf} (see Example~\ref{exinf}) it suffices to prove Theorem~\ref{thm:thmmain1}. Thus, in the sequel:
\begin{itemize}
\item
$c_i(x)$ and $f_i(x)$ are polynomials not satisfying  condition~\ref{iex} of Theorem~\ref{thm:thmmain} and 
$$
u_n(x)=c_1(x)f_1(x)^n+c_2(x)f_2(x)^n \qquad (n=1,2,\ldots); 
$$
 
\item
$m$ and~$n$ are positive integers  such that ${u_n(\zeta)=0}$ for a primitive $m$th root of unity~$\zeta$; since  ${u_n(x)\in \Q[x]}$, this is equivalent to 
\begin{equation}
\Phi_m(x)\mid u_n(x). 
\end{equation} 
\end{itemize}


\subsection{Some reductions}

We start by some general observations. 

\begin{itemize}

\item
We may assume that
\begin{equation}
\label{enonv}
c_1(\zeta)c_2(\zeta)f_1(\zeta)f_2(\zeta) \ne 0. 
\end{equation}
Otherwise ${\ph(m)\le D}$, and, using
\begin{equation}
\label{ephgeroot}
\ph(m)\ge m^{1/2} \qquad (m\ne 2,6) 
\end{equation}
(see~\cite{Va67}), we obtain ${m\le \max\{6, D^2\}}$, which is much sharper than what we want to prove. 

\item
We may assume that at least one of $f_1,f_2$ is a non-constant polynomial. Otherwise ${\deg u_n(x)\le D}$, and we again obtain ${\ph(m)\le D}$.

\item
We may assume that ${n>D}$. Otherwise ${\deg u_n(x) \le D+D^2}$, and, using~\eqref{ephgeroot} we obtain ${m\le \max\{6, (D+D^2)^2\}}$, again much sharper than the wanted result.

\item
Replacing  ${c_i(x)}$ and ${f_i(x)}$  by 
$$
\tilc_i(x):=c_i(x)/\gcd(c_1(x),c_2(x)), \qquad \tilf_i(x):=f_i(x)/\gcd(f_1(x),f_2(x)),
$$
respectively, we may assume that the polynomials $c_1,c_2$ are coprime in the ring $\Q[x]$, and so are $f_1,f_2$: 
\begin{equation}
\label{ecoprime}
\gcd(c_1(x),c_2(x))=\gcd(f_1(x),f_2(x)) =1. 
\end{equation}
Lemma~\ref{lhdivide} implies that 
$$
\height(\tilc_1,\tilc_2) \le \height(c_1,c_2)+(D+1)\log2 \le X+(D+1)\log2, 
$$
and similarly for ${\height(\tilf_1,\tilf_2)}$. Hence, to prove~\eqref{eupperm} in the general case, it suffices to prove
\begin{equation}
\label{eupperngam}
m\le e^{30D(X+D)}
\end{equation}
in the ``coprime case'', that is, assuming~\eqref{ecoprime}.

\end{itemize}

We distinguish several cases according to the nature of roots of our polynomials:

\begin{enumerate}
\item
$f_1(x)f_2(x)$ admits a root which is non-zero and not a root of unity;

\item
$f_1(x)f_2(x)$ vanishes at  a root of unity;

\item
$f_1(x)f_2(x)$ vanishes only at~$0$. 
\end{enumerate}

These cases are treated separately in the subsequent subsections.

\subsection{The polynomial $f_1(x)f_2(x)$ admits a root~$\gamma$ which is non-zero and not a root of unity}
\label{ssgamma}

By symmetry, we may  assume that~$\gamma$ is a root of $f_1(x)$. 
Since the statement of Theorem~\ref{thm:thmmain1} is invariant under multiplication of the polynomials $c_1,c_2$ by the same non-zero rational number, we may assume that the polynomial $c_1(x)$ is monic. Similarly, we may assume that $f_1(x)$ is monic.

Denote ${K=\Q(\gamma)}$. Then 
$$
d:=[K:\Q]\le D.
$$
Since ${X\ge 3}$, the right-hand side of~\eqref{eupperngam} exceeds  ${10^{30}D^9}$. 
Hence we may assume that 
$$
m>\max\{ 2^{d+1}, 10^{30}d^9\}. 
$$ 
Theorem~\ref{thschin} together with Proposition~\ref{pprim} implies now that there exists a prime~$\gerp$ of~$K$ such that  ${\nu_\gerp(\Phi_m(\gamma))>0}$ and
\begin{equation*}
m<\norm\gerp.
\end{equation*}
So we only have to bound $\norm\gerp$.

\subsubsection{The numbers~$\beta$ and~$\delta$}
We have ${f_2(\gamma)\ne 0}$ by~\eqref{ecoprime}. However, it it possible that ${c_2(\gamma)=0}$. Denote~$r$  the order of~$\gamma$ as a root of $c_2(x)$, and set
$$
\beta=\frac{c_2^{(r)}(\gamma)}{r!}, \qquad \delta=f_2(\gamma), 
$$
These are non-zero elements of the number field~$K$.

We claim that one of the following holds:
\begin{align}
\label{elocal}
\nu_\gerp(\alpha)&<0 \quad\text{for some coefficient $\alpha$ of~$c_1$ or~$f_1$ or~$c_2$ or~$f_2$};\\
\label{ebeta}
\nu_\gerp(\beta) &>0;\\
\label{edelta}
\nu_\gerp(\delta)&>0. 
\end{align}
Indeed, since ${\nu_\gerp(\Phi_m(\gamma))>0}$, there exists a primitive $m$th rooth of unity~$\zeta$ and a prime ${\gerP\mid \gerp}$ of the field $K(\zeta)$ such that 
$$
\nu_\gerP(\zeta-\gamma) >0. 
$$
Now, if~\eqref{elocal} does not hold, then  our four polynomials belong to ${\OO_\gerP[x]}$, where~$\OO_\gerP$ is the local ring  of~$\gerP$. Moreover, since~$f_1$ is monic, ${\gamma\in \OO_\gerP}$. Hence  the polynomials 
$$
F(x):=\frac{c_1(x)f_1(x)^n}{(x-\gamma)^r}, \qquad G(x) :=\frac{c_2(x)}{(x-\gamma)^r}
$$
belong to $\OO_\gerP[x]$ as well.  Note that $F(x)$ is indeed a polynomial, and moreover 
$$
F(\gamma) =0,
$$
because ${n>D\ge r}$. 

We have ${\beta=G(\gamma)}$ and 
${F(\zeta) =-G(\zeta)f_2(\zeta)^n}$  (because ${u_n(\zeta)=0}$). 
This implies the following congruences  in the ring~$\OO_\gerP$:
\begin{align*}
\beta\delta^n \equiv G(\zeta)f_2(\zeta)^n \equiv -F(\zeta)\equiv -F(\gamma) \equiv 0\mod\gerP. 
\end{align*}
Hence either ${\beta\equiv0\bmod \gerP}$ or ${\delta\equiv0\bmod \gerP}$, which means that one of~\eqref{ebeta} or~\eqref{edelta} holds true.

\subsubsection{Estimates}
Now we are ready to estimate $\norm\gerp$. Using Lemma~\ref{lvals}, we obtain 
\begin{equation}
\log\norm\gerp \le \max\{\height(1,c_1),\height(1,c_2),\height(1,f_1), \height(1,f_2), \height(\beta), \height(\delta)\}. 
\end{equation}
Since $f_1(x)$ is a monic polynomial, we have 
\begin{equation}
\label{ehfoneftwo}
\height(1,f_1),\height(1,f_2)\le \height(f_1,f_2) \le X, 
\end{equation}
and similarly for $c_1,c_2$. 
Furthermore, using Lemma~\ref{lhpol}, we find
\begin{align*}
\height(\gamma) & \le \height(f_1)+\log2\\
& \le X+\log2,\\
\height(\delta) &\le \height(1,f_2)+ D\height(\gamma) + \log(D+1)\\
&\le (D+1)X+2D,\\
\height(\beta) &\le \height(1, c_2^{(r)}/r!) + D\height(\gamma) + \log(D+1) \\
& \le \height(1,c_2) + D\log2 +DX+D\log2+\log(D+1) \\
&\le (D+1)X+2D. 
\end{align*}
This implies that 
$$
\log\norm\gerp \le (D+1)X+2D <3DX. 
$$
Since ${m< \norm\gerp}$, this proves~\eqref{eupperngam}.

\subsection{The polynomial $f_1(x)f_2(x)$ vanishes at  a root of unity~$\xi$}

We may assume that ${f_1(\xi)=0}$. Then ${f_2(\xi)\ne 0}$ by~\eqref{ecoprime}. 

Let us describe our argument informally. Since ${f_1(\xi)/f_2(\xi)=0}$, there exists ${\eps>0}$  such that ${|f_1(z)/f_2(z)|\le 1/2}$ when ${|z-\xi|\le\eps}$. 

Now assume that ${u_n(\zeta)=0}$ for some primitive  $m$th root of unity~$\zeta$.  Using~\eqref{enonv}, we may write 
\begin{equation}
\label{ealphanow}
0\ne\alpha:=\frac{c_2(\zeta)}{c_1(\zeta)}=-\left(\frac{f_1(\zeta)}{f_2(\zeta)}\right)^n. 
\end{equation}  

Let ${\Q(\zeta)\stackrel\sigma\hookrightarrow \C}$ be a complex embedding of the field~$\Q(\zeta)$ such that~$\zeta^\sigma$ belongs to the $\eps$-neighborhood of~$\xi$. Then
${|\alpha^\sigma| \le (1/2)^n}$. 
Define
\begin{equation}
\label{ebetanow}
\beta:=\prod_{|\zeta^\sigma-\xi|\le \eps} \alpha^\sigma,
\end{equation}
the product being over all~$\sigma$ as above. Since the $\eps$-neighborhood of~$\xi$  contains a positive proportion  of primitive $m$th roots of unity, we have
$$
-\log|\beta|\gg n\ph(m), 
$$
where the implied constant depends on  our polynomials $c_i$ and $f_i$ and on our choice of~$\eps$.

On the other hand,  ${\alpha\ne 0}$, and ${\height(\alpha) \ll1}$ by Lemma~\ref{lhpol}. Hence Liouville's inequality (Lemma~\ref{lliouv}) implies that 
$$
-\log|\beta| = \sum_{|\zeta^\sigma-\xi|\le \eps}-\log |\alpha^\sigma| \ll [\Q(\zeta):\Q]=\ph(m). 
$$
This bounds~$n$. 

This all will be made explicit in Subsection~\ref{sssexpl}.  But first, we establish some simple lemmas. 

\subsubsection{Some lemmas}


\begin{lemma}
\label{lcountcoprime}
Let ${a,b\in \R}$, ${a<b}$, and~$m$ a positive integer. Denote ${\ph(m,a,b)}$ the number of integers~$k$ coprime with~$m$ and satisfying ${a\le k\le b}$. Then 
$$
\ph(m,a,b) = (b-a)\ph(m)+O_1(2^{\omega(m)}).
$$
\end{lemma}

For the proof, see \cite[Lemma~2.3]{FGL17}. 

\begin{lemma}
\label{lcountroots}
Let~$\eps$ satisfy ${0<\eps\le 1}$ and let~$\xi$ be a complex number on the unit circle; that is,  ${|\xi|=1}$. Let~$m$ be a positive integer. Then there exist at least ${\pi^{-1}\eps\ph(m)- 2^{\omega(m)}}$ primitive $m$th roots of unity~$\zeta$ satisfying ${|\zeta-\xi|\le \eps}$. 
\end{lemma}

\begin{proof}
Write ${\xi=e^{2\pi \theta i}}$ with ${\theta\in \R}$, and let ${\eta>0}$ be the smallest positive real number with the  property ${2\sin (\pi \eta)=\eps}$. Note that ${1/6\ge \eta>(2\pi)^{-1}\eps}$. If~$k$ is an integer satisfying 
$$
m(\theta-\eta) \le k\le m(\theta+\eta), \qquad \gcd(m,k)=1,
$$
then ${\zeta:=e^{2\pi i k/m}}$ is a primitive $m$th root of unity satisfying ${|\zeta-\xi|\le \eps}$. 

Lemma~\ref{lcountcoprime} implies that there is at least ${2\eta\ph(m)- 2^{\omega(m)}}$ choices for~$k$, with  distinct~$k$ giving rise to distinct~$\zeta$ (this is because ${\eta\le 1/6}$). Since ${\eta \ge (2\pi)^{-1}\eps}$, the result follows. 
\end{proof}

\begin{lemma}
\label{leps}
Let ${f_1(x),f_2(x)\in \C[x]}$ be polynomials of degrees bounded by~$D$, and with coefficients bounded by ${H\ge 1}$ in absolute value. Let ${\xi\in \C}$ be such that 
$$
|\xi|\le 1, \qquad f_1(\xi)=0, \qquad f_2(\xi)=\delta\ne 0.
$$ 
Set 
$$
\eps= \frac{\min\{|\delta|,1\}}{3D^2H}. 
$$
Then for ${z\in \C}$ satisfying ${|z-\xi|\le \eps}$ we have ${|f_1(z)/f_2(z)|\le 1/2}$. 
\end{lemma}

\begin{proof}
Since ${|\xi|\le 1}$ and ${\eps \le 1/3D}$, we have for ${|z-\xi|\le\eps}$  trivial estimates
$$
|f_i'(z)|\le \frac12D(D+1)H(1+\eps)^{D-1}\le D^2H \qquad (i=1,2). 
$$
Hence for ${|z-\xi|\le\eps}$ we have 
\begin{equation*}
|f_1(z)|\le D^2H\eps \le \frac13|\delta|, \qquad
|f_2(z)|\ge |\delta|-D^2H\eps \ge \frac23|\delta|. 
\end{equation*}
This proves the lemma. 
\end{proof}

\subsubsection{The estimates}
\label{sssexpl}

As in Subsection~\ref{ssgamma} we may assume that~$f_1$ is monic, which implies that we have~\eqref{ehfoneftwo}. In particular,  the coefficients of~$f_1$ and~$f_2$ are bounded in absolute value by ${H:=e^X}$. Set ${\delta=f_2(\xi)}$. 

Note that the degree of~$\xi$ is at most~$D$ and the height is~$0$, because it is a root of unity. 
Using Lemmas~\ref{lhpol} and~\ref{lliouv}, we estimate
$$
|\delta|\ge e^{-\height(f_2(\xi))} \ge e^{-\height(1,f_2)-\log(D+1)}\ge ((D+1)H)^{-1}. 
$$
Setting ${\eps =(6D^3H^2)^{-1}}$, Lemma~\ref{leps} implies that 
$$
\left|\frac{f_1(z)}{f_2(z)}\right|\le 1/2
$$ 
for ${z\in \C}$ with  ${|z-\xi|\le \eps}$. 

Now define~$\alpha$ and~$\beta$ as in~\eqref{ealphanow},~\eqref{ebetanow}. Then 
\begin{equation}
\label{ebetasmall}
-\log|\beta|\ge nr\log2,
\end{equation}
where~$r$ is the number of embeddings ${\Q(\zeta)\stackrel\sigma\hookrightarrow\C}$ such that ${|\zeta^\sigma-\xi|\le \eps}$. Denote ${\sigma_1, \ldots, \sigma_r}$ all those~$\sigma$. 
Lemmas~\ref{lliouv} and~\ref{lhpol} imply that 
\begin{align*}
-\log|\beta|&=\sum_{i=1}^r-\log|\alpha^{\sigma_i}|\\
&\le [\Q(\zeta):\Q] \height(\alpha) \\
&\le \ph(m) (\height(c_1,c_2)+\log(D+1))\\
&\le \ph(m) (X+\log(D+1)). 
\end{align*}
Together with~\eqref{ebetasmall} this implies that 
\begin{equation}
\label{erphm}
n\le \frac{\ph(m)}{r\log2}(X+\log(D+1)),
\end{equation}
so we only have to bound~$r$ from below.

Lemma~\ref{lcountroots} implies that 
$$
r\ge \pi^{-1}\eps \ph(m)-2^{\omega(m)},  
$$
where we recall that  
${\eps=(6D^3H^2)^{-1}}$ with ${H=e^X}$.  
Using~\eqref{eboundphomega} with~$n$ replaced by~$m$, a messy but trivial calculation shows that either ${m\le e^{30D(X+D)}}$ (as we want) or  
${2^{\omega(m)} \le (2\pi)^{-1}\eps \ph(m)}$. 
Thus, ${r\ge (2\pi)^{-1}\eps \ph(m)}$, which, substituted to~\eqref{erphm}, gives 
$$
n\le 100 D^4e^{3X}. 
$$
Then 
$$
\ph(m) \le \deg u_n(x) \le 200D^5e^{3X},
$$
and, using~\eqref{ephgeroot}, we deduce from this an estimate much sharper than~\eqref{eupperngam}.

\subsection{The only root of $f_1(x)f_2(x)$ is~$0$}

We may assume that ${f_1(x)=1}$ and ${f_2(x)=\kappa x^b}$, where ${\kappa\in \Q^\times}$ and 
$$
1\le b\le D<n.
$$
We recall the following theorem of Mann~\cite{Ma65}. 

\begin{theorem}
Let ${a_0,a_1,\ldots,a_k\in\Q^\times}$ and $x_0=1,x_1,\ldots,x_k$ be roots of unity such that
\begin{equation}
\label{eq:Mann}
a_0x_0+a_1x_1+\cdots+a_kx_k=0.
\end{equation}
Assume that 
\begin{equation}
\label{eq:nondeg}
\sum_{i\in I} a_i x_i\ne 0
\end{equation}
for every non-empty proper subset ${I\subset \{0,\ldots,k\}}$. Then $x_i^m=1$
where 
$$
m=\prod_{p\le k+1} p.
$$
\end{theorem}

For us, we label
$$
c_i(x)=\sum_{j=0}^D c_{i,j} x^j\quad {\text{\rm for}}\quad  i=1,2,
$$
and we get
\begin{equation}
\label{eq:Mann1}
\sum_{j=0}^D c_{1,j} \zeta^j+\sum_{j=0}^D c_{2,j}\kappa^n \zeta^{j+nb}=0.
\end{equation}
This almost looks like the equation from Mann's theorem \eqref{eq:Mann} except that the non-degeneracy condition \eqref{eq:nondeg} might fail. So, let us study~\eqref{eq:Mann1}. Let $C$ be the set of non-zero coefficients among $c_{1,j}$ and $c_{2,j}\kappa^n$ for $0\le j\le D$. If $c\in C$ then ${c=c_{\ell,j}\kappa^{\delta n}}$ for some $\ell\in \{1,2\}$ and $j\in \{0,\ldots,j\}$, then put 
$x_c=\zeta^{j+\delta nb}$. Here, we take $\delta=0$ if $\ell=1$ and $\delta=1$ if $\ell=2$. With these  conventions, equation \eqref{eq:Mann1} is 
$$
\sum_{c\in C} cx_c=0.
$$
This splits into a certain number of non-degenerate equations. That is, there is a partition $C_1\cup C_2\cup \cdots \cup C_t=C$ such that $\sum_{c\in C_i} cx_c=0$ for ${i=1,\ldots,t}$ and each of these sub-equations is non-degenerate in the sense that it has no zero proper sub-sums. Clearly, ${\#C_i\ge 2}$ for each~$i$. 
We analyze two sub-cases.

\subsubsection{We have  $\#C_i\ge 3$ for some ${i\in \{1,\ldots,t\}}$}
\label{ssgethree}
Then $C_i$ contains two coefficients with the same $\ell$. We assume that ${\ell=1}$ (the case ${\ell=2}$ reduces to ${\ell=1}$ replacing~$\zeta$ by~$\zeta^{-1}$) and let $j_1<j_2$ be the smallest such that
$c_{1,j_1}$,~$c_{1,j_2}$ belong to $C_i$. Then the equation is 
$$
c_{1,j_1}\zeta^{j_1}+c_{1,j_2}\zeta^{j_2}+\sum_{\substack{c_{\ell,j}\kappa^{\delta n}\in C_i\\ \ell=2 ~{\text{\rm or}}~j>j_2}} c_{\ell,j} \kappa^{n\delta} \zeta^{j+n\delta b}=0.
$$
Dividing by $\zeta^{j_1}$, we get
$$
c_{1,j_1}+c_{1,j_2}\zeta^{j_2-j_1}+\sum_{\substack{c_{\ell,j}\kappa^{\delta n}\in C_i\\ \ell=2 ~{\text{\rm or}}~j>j_2}} c_{\ell,j} \kappa^{n\delta} \zeta^{j-j_1+n\delta b}=0.
$$
We are now in the position to apply Mann's theorem to conclude that 
$$
\zeta^{(j_2-j_1)m_1}=1,  \qquad m_1\mid \prod_{p\le \#C_i} p \mid  \prod_{p\le 2D+2}p,
$$
because ${\#C_i\le 2D+2}$. 
Since ${|j_2-j_1|\le D}$, we have 
\begin{equation}
\label{emzero}
m\le D\prod_{p\le 2D+2}p. 
\end{equation}
The inequality 
$
{\sum_{p\le x}\log p \le 1.02x} 
$
holds for all ${x>0}$, see \cite[Theorem~9]{RS62}. Hence 
$$
\log m\le \log D+\sum_{p\le 2D+2}\log p \le 4D, 
$$
which is much sharper than what we need.

\subsubsection{We have  $\#C_i=2$ for all $i=1,\ldots,t$} 

In fact, we may assume not only that $\#C_i=2$ but also that each $C_i$  contains exactly one $c_{1,j_1}$ and one $c_{2,j_2}\kappa^n$; otherwise the argument from Subsection~\ref{ssgethree} applies, and we again have~\eqref{emzero}. So, let 
$$
c_{1,j_1}\zeta^{j_1}+c_{2,j_2}\kappa^n\zeta^{j_2+nb}=0.
$$
We then get $\zeta^{j_2-j_1+nb}=-c_{1,j_1}/c_{2,j_2} \kappa^{-n}$. The pair $(j_1,j_2)$ depends on $i$. Assume first that, as we loop over~$i$, the differences $j_2-j_1$ are not the same over all $i$; that is, there are two values of $i$ corresponding to  
say $(j_1,j_2)$ and $(j_1',j_2')$ such that ${j_2'-j_1'\ne j_2-j_1}$. We obtain 
$$
\zeta^{(j_2-j_1)-(j_2'-j_1')}=\frac{c_{1,j_1}/c_{2,j_2}}{c_{1,j_1'}/c_{2,j_2'}}
$$
and the number on the right is a root of unity belonging to~$\Q$. Hence it is $\pm1$. The exponent on the left satisfies 
$$
0\ne \bigl|(j_2-j_1)-(j_2'-j_1')\bigr|\le 2D. 
$$
Hence ${m\le 4D}$, again better than wanted.

Now let us assume that ${j_2=j_1+a}$
with  the same~$a$ for all $i$.  In this case
$c_{2,j_1+a}=\lambda c_{1,j_1}$ with the same ${\lambda \in \Q^\times}$ holds for all the~$i$ as well. This makes the rational function $c_2(x)/c_1(x)$  equal to $\lambda x^{a}$, and so 
$$
u_n(x)=c_1(x)(1+\lambda\kappa^n x^{a+nb}).
$$
Since ${u_n(\zeta)=0}$ but ${c_1(\zeta)\ne 0}$, we must have ${1+\lambda\kappa^n \zeta^{a+nb}=0}$, which means that ${\lambda\kappa^n}$ is a root of unity, so~$\pm1$. Now we have two options: either both~$\lambda$ and~$\kappa$ are $\pm1$, or none is. The first option means that condition~\ref{iex} of Theorem~\ref{thm:thmmain} is satisfied, which is against our hypothesis. Hence ${\lambda\kappa^n=\pm1}$, but ${\lambda,\kappa\ne \pm1}$. 

We have clearly ${\height(\kappa)=\height(f_1,f_2)\le X}$ and ${\height(\lambda)=\height(c_1,c_2)\le X}$. Since~$\kappa$ is a rational number, distinct from~$0$ and from $\pm1$, its numerator or denominator (say, the former) is at least~$2$ in absolute value. It follows that the denominator of ${\lambda=\pm\kappa^{-n}}$ is at least $2^n$ in absolute value. But the denominator of~$\lambda$ cannot exceed ${e^{\height(\lambda)}\le e^X}$. We obtain ${2^n\le e^X}$, which implies ${n\le \log X}$. Hence 
$$
\ph(m)\le \deg u_n(x) \le D+D\log X, 
$$
which implies a much sharper estimate for~$m$ than the wanted~\eqref{eupperngam}. 

Theorem~\ref{thm:thmmain1} is proved. 

\section{Proof of Theorem~\ref{thsmallestn}}
\label{ssmallestn}

Let~$\zeta$ be an $m$th primitive root of unity such that the set  
\begin{equation}
\label{eprop}
\{n\in \Z_{>0}: \text{$u_n(x)$ is not identically~$0$, but ${u_n(\zeta)=0}$}\}
\end{equation}
is not empty. If 
${c_1(\zeta)f_1(\zeta)=c_2(\zeta)f_2(\zeta)=0}$
then set~\eqref{eprop} consists of all positive integers, and includes~$1$ in particular.

If, say, ${c_1(\zeta)f_1(\zeta)\ne 0}$, and set~\eqref{eprop} is non-empty, then 
$$
c_1(\zeta)f_1(\zeta)c_2(\zeta)f_2(\zeta)\ne0.
$$ 
Denoting 
$$
\eta=\frac{f_1(\zeta)}{f_2(\zeta)}, \qquad \theta=-\frac{c_2(\zeta)}{c_1(\zeta)},
$$
set~\eqref{eprop} consists of~$n$ with the property ${\eta^n=\theta}$. If~$\eta$ is a root of unity, then its order divides $2m$, and there exists a positive ${n\le 2m}$ such that ${\eta^n=\theta}$. If~$\eta$ is not a root of unity, then
${n=\height(\theta)/\height(\eta)}$. We have ${\height(\theta) \le X+\log(D+1)}$ by Lemma~\ref{lhpol}, and ${\ph(m)\height(\eta) \ge 2(\log\ph(m))^{-3}}$, see~\eqref{evout}. Hence 
$$
n\le m(\log m)^3(X+\log D). 
$$
Theorem~\ref{thsmallestn} is proved. 

\subsection*{Acknowledgements}

Yu.~B. was partially supported by the Indian Government SPARC Project P445.  F. L. was supported in part by Grant NUM2020 from the Wits CoEMaSS. Part of this work was done while F. L. was visiting the Max Planck Institute for Mathematics in Bonn from September 2019 to February 2020. He thanks this institution for its support, hospitality and excellent working conditions.

We thank Yann Bugeaud, Philipp Habegger, Alina Ostafe and Igor Shpar\-linski for helpful discussions. We also thank the referee for the encouraging report and many useful suggestions that helped us to improve the presentation. 

{\footnotesize

\bibliographystyle{amsplain}
\bibliography{ost}

}

\end{document}